\documentclass[12pt]{article}

\usepackage{amsmath}
  \usepackage{paralist}
  \usepackage{graphicx}
  \usepackage{amssymb}
  \usepackage{amsthm}
\usepackage{color}

\numberwithin{equation}{section}

\def \beq{\begin{equation}}
\def \eeq{\end{equation}}

\def \i{\infty}

\def \ep{\epsilon}

\def \a{\alpha}

\begin{document}
\begin{small}
\title{  Vanishing viscosity solution to a $2 \times 2$ system of conservation laws with linear damping}
\author{ Kayyunnapara Divya Joseph$^1$\\
 Department of Mathematics, \\
 Indian Institute of Science Education and Research Pune,\\
 Pune 411008, India 
}

\maketitle
\numberwithin{equation}{section}
\numberwithin{equation}{section}
\newtheorem{theorem}{Theorem}[section]
\newtheorem{remark}[theorem]{Remark}

\begin{abstract}
 Systems of the first order partial differential equations with singular solutions appear in many multiphysics problems and the weak formulation of the solution involves, in many cases, the product of distributions. In this paper, we study such a system derived from Eulerian droplet model for air particle flow. This is a 2 x 2 non - strictly hyperbolic system of conservation laws with linear damping. We first study a regularized viscous system with variable viscosity term, obtain a weak asymptotic solution with general initial data and also  get the solution in Colombeau algebra. We  study the vanishing viscosity limit and show that this limit is a distributional solution. Further, we study the large time  asymptotic behaviour of the viscous system. This  important system is not very well studied due to  complexities in the analysis.  As far as we know, the only work done on this system is for Riemann type of initial data. The significance of this paper is that we work on the system having  general initial data and not just initial data of the Riemann type. 
{ \bf ORCID:} Kayyunnapara Divya Joseph https://orcid.org/0000-0002-4126-7882 \\
{ \bf E-Mail: } divediv3@gmail.com \\  
{ \bf Ph. No.:} +917259149455
\end{abstract}

\section{Introduction} 
One of the main features in the study of solutions of the initial value problem for nonlinear first order evolution equations is the finite time break down of classical solution, even if the initial data is smooth. Continuing  the solution beyond this time and getting a global in time solution involves in many cases,  the product of distributions. The Schwartz impossibility result shows that, there is no algebra with some basic properties, that contain the space of distributions,  which allow such products. Since  the product of distributions naturally arise  in applications in physical problems, there are many attempts to give meaning to the products. For example, one of the unresolved problems is the solution of the continuity equation in $R^d$,
\[
\rho_t +\nabla.(u \rho)=0,
\]
where $u$ is the solution of the multi-dimensional inviscid Burgers equation 
\[
u_t +(u . \nabla) u=0
\]
in space variables $x \in R^d$ which involves the non-conservative product. This problem comes in the large scale structure formation in the universe, $u$ is velocity and $\rho$ the density. Here, the continuity equation is a linear equation with non-smooth coefficient $u$. We can expect atmost BV regularity for the coefficient $u$.  For works related this problem, we refer to the works of Shandarin and Zeldovich \cite{Shandarin}, S. N. Gurbatov and A.I. Saichev \cite{Saichev}, S. Albeverio and  V. M. Shelkovich \cite{Albeverio}.  For a general theory on non-conservative systems, we refer to the works, Colombeau
  \cite{Colombeau1,Colombeau2,Colombeau3}, Egorov \cite{Egrov}, Dal Maso, LeFloch,Murat \cite{Dalmaso}, Volpert \cite{Volpert}. \\

 In this paper we focus on three theories,  the method of weak asymptotics, the notion of Volpert product and the theory of Colombeau algebra, 
 in context of the construction of  the solution to a specific system of conservation laws with linear damping given by, 
\begin{equation}
 \begin{aligned}
&\frac{\partial u}{ \partial t} + \frac{\partial }{ \partial x} ( \frac{u^2}{2} ) + \a u=0, \\
&\frac{\partial v}{ \partial t} + \frac{\partial }{ \partial x} (u v) =0, 
\end{aligned}
\label{e1.1}
\end{equation}
where $\a \neq 0$ is a constant, with initial data of the form 
\begin{equation}
 \begin{aligned}
u(x, 0)= u_0(x), \,\,\, v(x, 0)= v_0(x).
\end{aligned}
\label{e1.2}
\end{equation}
When $\a>0$, the system \eqref{e1.1} is derived from the Eulerian droplet model  for air particle flow for smooth solutions, see  \cite{Keita}. Here, $\a$ is the drag coefficient between  air and the particles, $v \geq 0$ and $u$ are the volume fraction and velocity of the particles respectively.  The Riemann problem for the Euler droplet model, was studied in \cite{Keita} and that for \eqref{e1.1}, was studied in  \cite{Richard}. The solutions have  common features, namely the formation of delta-shocks in the solution.  The first equation of \eqref{e1.1} for $u$, has the effect of non-linear convection and linear damping. The solution space for $u$ is the space of functions of bounded variation. The second equation in \eqref{e1.1} for $v$, is a linear equation with a non-smooth coefficient and the solution space is the space of bounded Borel measures. In fact, this system has repeated eigenvalue $u$ of multiplicity $2$ with one dimensional eigenspace, so  it is a  non-strictly hyperbolic system.  So the standard  theory of strictly-hyperbolic systems in \cite{Glimm, Lax} does not work. For the mathematical analysis of the system \eqref{e1.1}, we need to use new formulations of generalised solutions, which is not well developed. The main aim of this paper is to construct the explicit solution for the problem \eqref{e1.1} -\eqref{e1.2} in  different frameworks of weak solutions. In this construction of solutions, the vanishing viscosity approximations play an important role.

When $\a =0$, the  first equation in system \eqref{e1.1} is commonly called the Burger's equation that  was first derived in 1915 by Bateman  \cite{Bateman}, later reinvented by Burgers  \cite{Burgers} in 1948.  It is useful in understanding various phenomenon like shock waves in gas dynamics, see the work of Hopf \cite{Hopf}. In fact,
when $\a =0,$ $u$ is constant along the characteristics defined by $\frac{dx}{dt}=u$  and these characteristics are straight lines.  For non-increasing initial data,  the characteristics starting at different points of  initial line meet  at some later time and shocks are generated. \\

The system \eqref{e1.1}  with $\alpha=0$ is used in modelling the evolution of density inhomogeneities in matter in the universe, see  Shandarin and Zeldovich \cite{Shandarin}. 
There are more recent applications involving kinetic models of stochastic production flows, such as the flows of products through a factory or supply chain. When expanded into deterministic moment equations, these kinetic models lead to  system \eqref{e1.1} with $\a=0$, see \cite{a1, a2,  Forestier}. The particles  moving along the characteristics lines collide,  stick together, become massive particles and due to this concentration phenomenon, $v$ generally becomes  a measure. 
Systems of this type  were analysed by many authors, see S. N. Gurbatov and A.I. Saichev \cite{Saichev}. S. Albeverio and  V. M. Shelkovich \cite{Albeverio}  worked on this system of zero pressure gas dynamics by finding the specific  type of solutions for Cauchy problem with piecewise-smooth
initial data. They proved that, these are $\delta$ waves solutions  which are related with
the concentration process on the surface which carries the singularities.  In the papers \cite{Joseph1,Joseph2,LeFloch, Tan},  the well-posedness of the Cauchy problem to  this system \eqref{e1.1} with $\alpha=0$ was  analyzed. In all these papers, the study  of the system  generally involves very complex analytical techniques. When $\a\neq 0$, we have
the same phenomenon, but the characteristics are not straight lines and  $u$ is not constant along the characteristics. Indeed, $\frac{du}{dt}$ along the characteristics is $- \a u$, since   along the characteristic curve   
\begin{equation*}
 \frac{d u(x(t), t)}{dt} = u_t + u_x \frac{dx}{dt}=  u_t + u u_x = -\a u.
\end{equation*}
 The exploration of new ideas for the general system \eqref{e1.1},  make  works   on this system  very important.  Richard de la Cruz \cite{Richard} studied the Riemann problem to \eqref{e1.1}. In the Riemann problem,  the initial data $u_0,v_0$ are of the form
\begin{equation*}
u_0(x)= u_l + (u_r -u_l)H(x),\,\,\,v_0(x)= v_l + (v_r -v_l)H(x),
\end{equation*}
where $H(x)$ is the Heaviside function and  $u_l ,u_r, v_l, v_r$ are constants. He used a vanishing viscosity approximation  with variable viscosity coefficient of the form

\begin{equation}
 \begin{aligned}
&\frac{\partial u}{ \partial t} + \frac{\partial }{ \partial x} ( \frac{u^2}{2} ) + \a u =  {\ep} e^{- \a t}  \frac{{\partial}^2 u}{ \partial x^2}, \\
&\frac{\partial v}{ \partial t} + \frac{\partial }{ \partial x} (u v) =  {\ep} e^{- \a t}  \frac{{\partial}^2 v}{ \partial x^2},
\end{aligned}
\label{e1.3}
\end{equation}
and used the generalized Hopf Cole transformation to get the  solution for the approximation in explicit form. By letting $\ep \rightarrow 0$, he got an  explicit formula for the Riemann problem to \eqref{e1.1}. 

In our work, we get  the exact solution to the Cauchy problem to system \eqref{e1.1} with general initial condition \eqref{e1.2} in different spaces. First, we review different notions of solution based on the weak asymptotic method, Colombeau algebra, the Volpert product and distribution theory. All these notions of solution has an underlining physical regularization and we use the viscous approximation \eqref{e1.3}.  We construct  the regularized problem explicitly and get the solution in the weak asymptotic sense, the weak solution in Colombeau algebra and the solution to \eqref{e1.1} in a distributional sense  with  general initial data \eqref{e1.2}. While doing so, we get an explicit formula for  the vanishing viscosity limit to this system and the large time behaviour of the viscous system \eqref{e1.3} with general initial data. This regularisation is justified, since we show the vanishing viscosity solution satisfies the Lax entropy inequality for speed $u$. The significance of the present work is due to the consideration of general initial data, for the system \eqref{e1.1} which makes the analysis  complicated and it is not yet attempted before. This gives the work its importance.  \\

 We do not have a uniqueness result for the initial value problem for \eqref{e1.1}. Uniqueness of the $u$ component of the solution that we constructed  follows by Kruzkov theory, see \cite{Kruzkov}. Since $v$ is a measure and satisfies a linear equation with a discontinuous coefficient,  the standard method of $L^1$ estimate is not applicable and the uniqueness of $v$ remains an open problem.  One of the physical selection criteria for the solution  is by the vanishing viscosity limit and we used this method in this paper. \\

 The rest of the paper is organised in the following way. In section 2, we introduce different  formulations of the weak solutions. In section 3, we construct the approximate solution obtained by a parabolic regularisation and  we study  the large time  asymptotic behavior. In section 4, we construct the weak solution to the Cauchy problem and the last section provides a  conclusion to our study.

\section {Weak formulation of solutions} In this section, we review the different notions of solutions based on the weak asymptotic method, the Volpert product and Colombeau algebra. We start with the weak asymptotic method.

\subsection {Weak asymptotic method}
The weak asymptotic method has its root in the works of Maslov \cite{Maslov}. This method is highly sucessful in the study of nonlinear waves, see \cite{Albeverio,  Danilov, Egrov, Shelkovich} for discussions on this topic and many important applications. We say that a family of smooth functions $f^\epsilon$ is a regularization of a distribution $f$ if $f^\epsilon$ converges to $f$ in distribution as $\epsilon$ goes to zero. The idea is to construct approximate regularized solutions, which involve a positive parameter $\epsilon$ going to zero, and satisfy the differential equation upto $O(\epsilon^j)$, $j$ being a positive number uniformly in time $t>0$. In the context of the system \eqref{e1.1}-\eqref{e1.2}, we have the following definition. 
{ \flushleft
\textbf{Definition 2.1.}  a) A family of smooth functions $\left( u^{\epsilon}, v^{\epsilon} \right)_{\epsilon >0}$ defined on $R^1 \times [0,\infty)$ is called a \textit{weak asymptotic solution} to the system
\eqref{e1.1} with initial conditions \eqref{e1.2} provided: \\
i)  For all  $\phi \in C_c^{\infty}(-\infty, \infty)$ and  for each $T>0$,
\begin{equation}
\begin{aligned}
\int_R{ ( u^\epsilon_t(x, t) + u^\epsilon(x, t)  u^\epsilon_x(x, t) + \a u^{\ep}(x, t)) \phi(x) dx} &=o(1),\\
\int_R{ ( v^\epsilon_t(x, t)  + {(v^\epsilon(x, t)  u^\epsilon(x, t) )}_x ) \phi(x)  \,\,\, dx} &= o(1),\\  
\end{aligned}
\label{e2.1}
\end{equation}
holds as $\epsilon \rightarrow 0, $ uniformly with respect to  $t\in [0, T]$,  i.e. \eqref{e2.1} holds for all $t\in [0, T]$ as $\epsilon \rightarrow 0.$   \\
ii) As $\epsilon \rightarrow 0,$ for all  $\phi \in C_c^{\infty}(0, \infty),$  
\begin{equation}
\begin{aligned}
\int_R{ ( u^\epsilon(x,0)-u_0(x) ) \phi(x)  dx} &=o(1),\\
\int_R {( v^\epsilon(x,0)- v_0(x) ) \phi(x)  dx} &=o(1),
\end{aligned}
\label{e2.2}
\end{equation}
must  holds. \\
}
{ \flushleft
\textbf{Definition 2.2.} A distribution $(u,v)$ is called a \textit{generalized  solution} for the Cauchy problem \eqref{e1.1},\eqref{e1.2},
if it is the distributional limit of a weak asymptotic solution $(u^{\epsilon},v^{\epsilon})_{\epsilon >0}$ as $\epsilon \rightarrow 0$. \\
}
\vspace{0.25 cm}
We refer to \cite{Albeverio, Danilov, Joseph3, Shelkovich} and the references therein for the theory and many applications of this theory.

\subsection{Functions of bounded variation and Volpert product}
In order to define the notion of distributional solutions, we should make sense of each term in the equation. The main difficulty is to define the term $uv$ in the equation. This cannot be done for arbitrary distributions. 
In  Volpert theory, the product of distributions is defined when they have some regularity properties.
Here, we recall some basic facts  in the theory of functions of bounded variation from \cite{Volpert, Dalmaso}.
Let  $\Omega$ be an open set in $R^n$ and let $BV(\Omega, R^1)$ be the space of all Lebesgue integrable functions from $\Omega$  to $R^1$
whose first order distributional derivatives $\partial_{x_j}$ are bounded Borel measures on $\Omega$. This space is called {\bf  the space of  functions of bounded variation} from $\Omega$ to $R^1$. If $u \in BV(\Omega,R^1)$, the distributional gradient of $u$ is a vector valued  bounded Borel measure and its   {\bf  total variation} is defined by
\[
||Du||=Sup\{ \int_\Omega u \nabla. \phi  dx : \phi \in C_c^1(\Omega,R^n), ||\phi ||_{\infty}\leq 1 \}.
\]
 $BV(\Omega, R^1)$ is a Banach space with the norm defined by
\[
||u||_{BV(\Omega,R^p)} =||u||_{L^1(\Omega)}+||Du||.
\]
Similarly, we can define $BV(\Omega,R^p)$ as the space of functions of bounded variation from $\Omega$ to $R^p$ and 
$u =(u_1,...u_p) \in BV(\Omega,R^p)$ iff $u_j \in BV(\Omega,R^1)$.
 Let $u \in L^1(\Omega,R^p)$, we say a point $y_0\in \Omega$ is a Lebesgue point or a point of approximate continuity if there exists a vector $\bar{u}(y_0)$ in $R^p$ such that
\[
\frac{1}{\rho^n}\int_{B_{\rho}(y_0)}|u(y)-\bar{u}(y_0)|dy=0.
\]
A point $y_0\in \Omega$ is a point of approximate jump for $u$ if there exists two vectors $u_{-}(y_0), \,\,\, u_{+}(y_0)$ in $R^p$ and a unit vector $\nu_u=({\nu_{u}}^j)_{1 \leq j \leq n}$ in $R^n$ such that $u_{-}(y_0)\neq u_{+}(y_0)$ and
\[
\frac{1}{\rho^n}\int_{{B_\rho}^{\pm}(y_0)}|u(y)-\bar{u}_{\pm}(y_0)|dy=0.
\]
Here 
$B_\rho(y_0)= \{y \in \Omega: |y-y_0|<\rho \}$, the disc in $R^n$ of radius $\rho$ with center $y_0$ and ${B_\rho}^{\pm}(y_0)= \{y \in B_\rho(y_0) : \pm(y-y_0,\nu_u(y_0))>0  \}$ are half discs.

The triplet $(u_{-}(y_0),u_{+}(y_0),\nu_u(y_0))$ is uniquely determined up to a change of sign of $\nu_u(y_0)$.
 For the general $u \in L^1(U,R^p)$, one can write 
\[
\Omega=\Omega_u \cap S_u \cap I_u,
\]
where $\Omega_u$ is the set of points of approximate continuity, $S_u$ is the set of points of approximate jump and $I_u$ contains the irregular points.
Volpert proved that for $u \in BV(\Omega,R^p)$, 
\[
H^{n-1}(I_u)=0,
\]
where $H^{n-1}$ is the $n-1$ dimensional Hausdorff measure.
The functions $u_{\pm} :S_u \rightarrow R^p$ and $\nu_u :S_u \rightarrow S^{n-1}$ are measurable with respect to the Hausdorff measure $H^{n-1}$. The basic result on Volpert product is the following theorem, we state the simpler form of Volpert from \cite{Dalmaso}.\\
\begin{theorem}
Let $u \in BV( \Omega,R^p)$  and $g: R^p \rightarrow R^p$ be a locally bounded Borel function.  Then, there exists  a unique family of real valued bounded Borel measures $\mu_j$ on $\Omega, j=1,...n$  characterized by  the three properties:
\begin{itemize}
\item If $B$ is a Borel subset of $\Omega_u$, then
\[
\mu_j(B)=\int_B g(u)\partial_{x_j}u.
\]
\item If B is a Borel subset of $S_u$, then
\[
\mu_j(B)=\int_B \int_0^1 g((1-s) u_{-}(y)+s(u_{+}(y)-u_{-}(y)(u_{+}(y)-u_{-}(y))ds {
\nu_{u}}^j dH^{n-1}(y).
\]
\item If $B$ is a Borel subset of $I_u$,
\[
\mu_j(B)=0.
\]
\end{itemize}
\end{theorem}
{ \flushleft
\textbf{Definition 2.3.} The measure $\mu_j$ is called the non-conservative product of $g(u)$ and $\partial_{x_j}u$.\\
\vspace{0.25 cm}

This result is used in the definition of the the term $uv$ in \eqref{e1.1} and in the formulation of the weak solution in the sense of distributions, which we shall explain now. In our case, $\Omega =\{(x,t) : x \in R, t>0 \}$ and $p=2$.
 Let $(u,V)$ be a pair of functions of bounded variation.  Then  we know that, the distribution $v=V_x$ is a bounded Borel measure. \\
\vspace{0.25 cm}
\textbf{Definition 2.4.} Let $(u,v)$ be a distribution, where $u$ is a function of bounded variation and $v=V_x$ with $V$ as a function of bounded variation. $(u,v)$ is called a distributional solution for \eqref{e1.1} if 
\[
u_t +(u^2/2)_x +\alpha u=0,\,\,\, v_t +(\bar{u}v)_x=0
\]
in the sense of distribution, namely 
for all $\phi \in C_c^1(\mathbb{R} \times(0,\infty))$,
\[
\begin{aligned}
&\int_{R^1\times(0,\infty)}  u(x,t) \phi_t(x,t)+\frac{(u(x,t)^2)}{2} \phi_x( x,t)  - \alpha u(x,t) \phi(x,t) dx dt =0,\\
&\int_{R^1\times(0,\infty)} v \phi_t(x,t) dx dt+\int_{R^1\times(0,\infty)} \phi_x (\bar{u} v) =0.
\end{aligned} 
\]
}
\vspace{0.25 cm}

The formulation of the distributional solution for the first equation is standard. For the second equation, the term $\bar{u} v=\bar{u} V_x$ is a measure called the Volpert product of $u$ and $V_x$  given by Theorem 2.1. It  has an absolutely continuous part and a singular part where $\bar{u}$ is the averaged superposition of $u$, which we shall explain now.

 Since  $u$ and $V$ are  functions of bounded variation,  we  decompose the domain of definition  with respect to 
$(u,V)$, with a slightly different notation than introduced in Theorem 2.1:
\[
 \mathbb{R} \times[0, \infty) =S_c \cup S_j \cup S_0
\]
where $S_c$  are points of approximate continuity of $(u,V)$, $S_j$ 
are points of approximate jump of $(u,V)$ and $S_0$ is a set with one dimensional 
Hausdorff-measure zero. At any point $(x,t) \in S_j$, the left limit   $(u(x-, t),V(x-, t))$ and 
the right limt $(u(x+, t),V(x+, t))$  of $(u(x,t),V(x,t))$ exist. Further, for any 
continuous  real valued function $g$ on $R^1,$  the Volpert product 
$g(u)v =g(u)V_x$ is  defined as a Borel measure using the averaged superposition of $g(u)$  defined in Volpert \cite{Volpert} as follows:

\begin{equation}
\overline{g(u)}(x,t) =
 \begin{cases} \displaystyle
      {g(u(x,t),\,\,\, if \,\,\,(x,t) \in S_c,} 
\\\displaystyle
        { \int_0^1 g((1-\alpha)(u(x-,t)+\alpha u(x+,t))d\alpha
                      ,\,\,\, if (x,t) \in S_j }.
\end{cases}
\label{e2.3}
\end{equation}
For  any Borel subset $A$ of $S_c$,
\[
[g(u)V_x](A)=\int_{A}\overline{g(u)}(x,t)V_x
\]
with $\overline{g(u)}$ as in \eqref{e2.3}. If $(x,t) \in S_j$,
\[
[g(u)V_x](\{(x,t)\})=\overline{g(u)}(x,t)(V(x+0,t)-V(x-0,t)).
\]
Any point $(s(t),t) \in S_j$,  is a point on the  curve of discontinuity $x=s(t)$ of $(u, V)$.  Note that $s(t)$ is given by the Rankine Hugoniot condition 
for the equation for $u$:
\[
\frac{ds}{dt}= \frac{u(s(t)+,t)+u(s(t)-,t)}{2}.
\]
and hence $s(t)$ is a Lipschitz function. For a Borel subset $A$ of $S_j$, $[g(u)V_x](A)$ is an integral of $[g(u)V_x](\{(s(t),t)\})$ with respect to  $t$ over the set. The measure $g(u)V_x$ is zero on $S_0$.\\

We use these calculations to prove the existence of weak solutions of \eqref{e1.1}-\eqref{e1.2} in the sense of distributions.

\subsection{Solution in the sense of Colombeau}
In this section, we consider a larger class of functions as initial data, adopt the approach of Colombeau \cite{Colombeau1,Colombeau2,Colombeau3} and
construct the solution  of \eqref{e2.1} and \eqref{e2.2} with equality
replaced by association in the sense of Colombeau. This approach takes
into account not only the final limit but also the microscopic structure of 
the shock, due to the viscous effects in the solutions.

First, we describe the algebra of generalized functions of Colombeau in 
$\Omega = \{(x,t), x\in R^1, t>0\}$ denoted
by ${\mathcal G}(\Omega)$. Let $C^\infty(\Omega)$ be the class of 
infinitely differentiable functions 
in $\Omega$ and consider the infinite product ${\mathcal E}(\Omega)=
[C^\infty(\Omega)]^{(0,1)}$.
Thus, any element $v$ of ${\mathcal E}{(\Omega)}$ is a map from $(0,1)$ 
to
$C^\infty(\Omega)$
and is denoted by $v=(v^\epsilon)_{0<\epsilon<1}$. An element 
$v=(v^\epsilon)_{0<\epsilon<1}$ is called moderate, if  given a compact
subset K of $\Omega$
and $j, \,\,\,  \ell$ non negative integers, there exists $N>0$ such that
\begin{equation}
\parallel \partial^j_t \partial^{\ell}_x v^\epsilon
\parallel_{L^\infty
(K)} ={\mathcal {O}}(\epsilon^{-N}) ,
\label{e2.4}
\end{equation}
as $\epsilon$ tends to $0$. An element $v=(v^\epsilon)_{0<\epsilon<1}$ 
is called null,
if for all compact subsets K of $\Omega$, for all nonnegative integers
$j,\,\, \ell$ and for all $M>0$,
\begin{equation}
\parallel \partial^j_t \partial^{\ell}_x v^\epsilon \parallel_
{L^\infty(K)} = {\mathcal {O}}(\epsilon^{M}) ,
\label{e2.5}
\end{equation}
as $\epsilon$ goes to $0$. The set of all moderate elements is denoted
by ${\mathcal E_M}(\Omega)$ and the set of null elements is denoted by 
${\mathcal N}(\Omega)$. It is easy to
see that ${\mathcal E_M}(\Omega)$ is an algebra with partial 
derivatives, the
operations being defined pointwise on representatives and
${\mathcal N}(\Omega)$
is an ideal which is closed under differentiation. The quotient space 
denoted
by
\begin{equation*}
{\mathcal G}{(\Omega)} =\frac{{\mathcal 
E_M}(\Omega)}{{\mathcal 
N}(\Omega)}
\end{equation*}
is an algebra with partial derivatives, the operations being defined on
representatives. The algebra ${\mathcal G}(\Omega)$ is called the 
algebra of generalized functions of Colombeau.

{ \flushleft
\textbf{Definition 2.5.}
\begin{enumerate}
\item Two elements $u$ and $v$ in ${\mathcal G}(\Omega)$
are said to be { \it associated,} if for some (and hence all) representatives
$(u^\epsilon)_{0<\epsilon<1}$ and $(v^\epsilon)_{0<\epsilon<1}$, of $u$ 
and
$v$ , $u_\epsilon -v_\epsilon$ goes to $0$ as $\epsilon$ tends to $0$
in 
the
sense of distribution and is denoted by "$u \approx v$". \\
\item  We say $(u,v) \in {\mathcal G}(\Omega) \times {\mathcal G}(\Omega)$ is { \it a weak solution of \eqref{e1.1}   in the sense of association,} if
\begin{equation}
\begin{aligned}
u_t + (\frac{u^2}{2})_x +\alpha u& \approx 0\\
v_t + (u\rho)_x & \approx 0.
\end{aligned}
\label{e2.20}
\end{equation}
\end{enumerate}
}
Here we remark 
that,
this notion is different from the notion of equality in ${\mathcal 
G}(\Omega)$, which means that $u-v \in {\mathcal N}(\Omega)$ or in 
other words,
\[
\parallel \partial^j_t \partial^{\ell}_x (u^\epsilon-v^\epsilon)
 \parallel_{L^\infty(K)} = {\mathcal {O}}(\epsilon^{M})
\]
for all M, for all compact
subsets K of $\Omega$ and for all $j, \,\,\, \ell$ non-negative integers.\\
We refer to the works \cite{Colombeau1,Colombeau2,Colombeau3,Dalmaso, Volpert, Egrov} and the references 
therein,
that use Colombeau algebra to find global solutions 
of the initial value problems when the non-conservative product appears
in the equation.

In the next section, we obtain an explicit   solution to \eqref{e1.3}  with general initial conditions \eqref{e1.2} and 
use it  as the regularized solution to construct  a  weak  solution to the inviscid system \eqref{e1.1}-\eqref{e1.2} for a general class of initial data.

\section{\bf Construction of regularized approximate solution.}

In this section, we consider the regularized  system 

\begin{equation}
 \begin{aligned}
&\frac{\partial u}{ \partial t} + \frac{\partial }{ \partial x} ( \frac{u^2}{2} ) + \a u =  {\ep} e^{- \a t}  \frac{{\partial}^2 u}{ \partial x^2}, \\
&\frac{\partial v}{ \partial t} + \frac{\partial }{ \partial x} (u v) =  {\ep} e^{- \a t}  \frac{{\partial}^2 v}{ \partial x^2},
\end{aligned}
\label{e3.1}
\end{equation}
in the space - time domain $\{(x,t) : x \in \mathbb{R}, t>0\}$ with initial conditions
\begin{equation}
 \begin{aligned}
u(x, 0)= u_0(x), \,\,\, v(x, 0)= v_0(x), x \in \mathbb{R}
\end{aligned}
\label{e3.2}
\end{equation}
and we construct an explicit solution to this Cauchy problem.

\subsection{Explicit formula for the approximate solution}

Before stating the results, we introduce some functions which appear in the analysis of solution.
Fix $(x, t)$ for  $x \in \mathbb{R}, t>0$ and given $u_0 \in L^{\infty}(\mathbb{R})$,  set 
\begin{equation}
 \begin{aligned}
\theta(x, y, \tau) = \frac{{(x-y)}^2}{ 2\tau(t)} + \int_0^y{ u_0(s) ds} , \,\,\,\tau(t)=\frac{1-e^{- \a t}}{\alpha}. 
\end{aligned}
\label{e3.3}
\end{equation}
With these notations, we state the main existence results of this section.
\begin{theorem}
Let $u_0 \in L^\infty( \mathbb{R})$ and $v_0$  is in $L^1( \mathbb{R})$. Then, the following functions 
\begin{equation}
u^\ep(x,t)=e^{-\a t}{     \frac{  \int_{-\i}^{\i}{ \frac{(x-y)}{\tau(t)} e^{- \frac{\theta(x,y,\tau(t))}{2\ep} } dy}     }{      \int_{-\i}^{\i}{  e^{- \frac{\theta(x,y,\tau(t))}{2\ep} } dy}  } }, 
V^\ep(x,t)={     \frac{  \int_{-\i}^{\i}(\int_0^y v_0(s) ds )  e^{- \frac{\theta(x,y,\tau(t))}{2\ep} } dy     }{      \int_{-\i}^{\i}{  e^{- \frac{\theta(x,y,\tau(t))}{2\ep} } dy}  } }
\label{e3.4}
\end{equation}
are well defined  $C^\infty$ functions of $(x,t),$  for $x\in  \mathbb{R}, t>0$ and the pair of functions $(u^\ep,v^\ep)$ with $v^\ep=\partial_x V$ is the unique solution to \eqref{e3.1} with initial condition \eqref{e3.2}. \\
\end{theorem}
\begin{proof}
 Note that if $(\hat{U},\hat{V})$ solves
\begin{equation}
 \begin{aligned}
&\frac{\partial \hat{V}}{ \partial t} + e^{-\a t}\frac{\partial  \hat{U}}{ \partial x}  \frac{\partial \hat{V} }{ \partial x} =  {\ep} e^{- \a t}  \frac{{\partial}^2  \hat{V}}{ \partial x^2}  , \\
&\frac{\partial   \hat{U}}{ \partial t} + \frac{e^{-\a t}}{2}(\frac{\partial  \hat{U}}{ \partial x } )^2  =  {\ep} e^{- \a t}  \frac{{\partial}^2  \hat{U}}{ \partial x^2}
\end{aligned}
\label{e3.5}
\end{equation}
with initial data 
\begin{equation}
\hat{U}(x, 0)=\int_0^x u_0(s) ds,\,\,\,\hat{V}(x, 0)=\int_0^x v_0(s) ds,
\label{e3.6}
\end{equation}
then 

\begin{equation}
 \begin{aligned}
(u^\epsilon,v^\epsilon)=(e^{-\a t} \hat{U}_x,\hat{V}_x )
\end{aligned}
\label{e3.7}
\end{equation} 
solves \eqref{e3.1} with initial condition \eqref{e3.2}.

We use the Hopf-Cole transformation  \cite{Hopf, Joseph1,Richard},
\begin{equation}
 \begin{aligned}
\hat{U}= - 2 {\ep} \log{S^{\ep}}, \,\,\,
\hat{V} = \frac{C^{\ep}}{S^\ep}. 
\end{aligned}
\label{e3.8}
\end{equation} 
to reduce the problem \eqref{e3.5}-\eqref{e3.6} to  the system
\begin{equation*}
 \begin{aligned}
  e^{ \a t} \frac{\partial C^{\ep}}{ \partial t}  =  {\ep}  \frac{{\partial}^2 C^{\ep}}{ \partial x^2}, \,\,\,
e^{ \a t} \frac{\partial S^{\ep}}{ \partial t}  =  {\ep}  \frac{{\partial}^2 S^{\ep}}{ \partial x^2}  
\end{aligned}
\end{equation*}
with initial conditions
\[
S^\ep(x,0)=\exp{\{^{-\frac{\int_0^x u_0(s)ds}{2 \ep}}}\},\,\,\,C^\ep(x,0)=\int_0^x v_0(s) ds . \exp{\{^{-\frac{\int_0^x u_0(s)ds}{2 \ep}}\}}.
\]
Using  $ \tau = \frac{1- e^{ - \a t}}{ \a } >0$, $C^\ep(x,t)=\tilde{C}^{\ep}(x,\tau),  S^\ep(x,t)=\tilde{S}^{\ep}(x,\tau)$,  our problem is reduced to solving
\begin{equation*}
 \begin{aligned}
   \frac{\partial \tilde{C}^{\ep}}{ \partial { \tau}}  =  {\ep}  \frac{{\partial}^2 \tilde{C}^{\ep}}{ \partial x^2}, \,\,\,
 \frac{\partial \tilde{S}^{\ep}}{ \partial { \tau}}  =  {\ep}  \frac{{\partial}^2 \tilde{S}^{\ep}}{ \partial x^2}  
\end{aligned}
\end{equation*}
with initial conditions 
\begin{equation*}
 \begin{aligned}
\tilde{S}^{\ep}(x, 0)= \exp{\{-\frac{\int_0^x u_0(s)ds}{2 \ep}}\},\,\,\,\tilde{C}^\ep(x,0)=\int_0^x v_0(s) ds . \exp{\{^{-\frac{\int_0^x u_0(s)ds}{2 \ep}}\}}.
\end{aligned}
\end{equation*}
Solving for $(\tilde{S^\ep}, \tilde{C^{\ep}})$, we obtain
\begin{equation*}
 \begin{aligned}
&\tilde{C}^{\ep}(x, \tau)= \frac{1}{ \sqrt{4 \pi \tau \ep }} \int_ \mathbb{R}{ \exp\{ \frac{-{(x-y)}^2}{4 \tau \ep}  \} }
 ( \int_0^y v_0(s) ds ) \exp{\{^{-\frac{\int_0^y u_0(s)ds}{2 \ep}}\}}dy,\\
&\tilde{S}^{\ep}(x, \tau)= \frac{1}{ \sqrt{4 \pi \tau \ep }} \int_ \mathbb{R}{ \exp\{ \frac{-{(x-y)}^2}{4 \tau \ep}  \} \exp{\{-\frac{\int_0^y u_0(s)ds}{2 \ep}}\}  dy}.
\end{aligned}
\end{equation*}
We rewrite this formula using $\theta(x,y, \tau)$, to get
\begin{equation}
 \begin{aligned}
&\tilde{C}^{\ep}(x, \tau)= \frac{1}{ \sqrt{4 \pi \tau \ep }} \int_ \mathbb{R} { (\int_0^y v_0(s) ds)  \exp\{- \frac{\theta(x,y,\tau)}{2  \ep}  \}  }dy,\\
&\tilde{S}^{\ep}(x, \tau)= \frac{1}{ \sqrt{4 \pi \tau \ep }} \int_ \mathbb{R}{ \exp\{- \frac{\theta(x,y,\tau)}{2  \ep}  \} } dy.
\end{aligned}
\label{e3.9}
\end{equation}
 Now using \eqref{e3.7} and \eqref{e3.8},  we can express $(u^\ep, v^\ep)$ in terms of  $\tilde{C}^\ep, \tilde{S}^\ep$ and their first
$x$ derivatives as follows:
\begin{equation}
 \begin{aligned}
u^{\ep}= -2 \ep e^{ - \a t} \frac{  \frac{ \partial \tilde{S}^{\ep}}{\partial x}}{ \tilde{S}^{\ep}}, \,\,\,  
v^{\ep}= \frac{ \partial }{ \partial x} ( \frac{ \tilde{C}^{\ep} }{ \tilde{S}^{\ep}} ).   
\end{aligned}
\label{e3.10}
\end{equation}
Since the integrals are absolutely convergent, we can differentiate under the integral sign and get
\begin{equation}
 \frac{ \partial \tilde{S}^{\ep}}{\partial x} =-\frac{1}{2 \ep} \frac{1}{ \sqrt{4 \pi \tau \ep }} \int_ \mathbb{R} \frac{(x-y)}{\tau}{ \exp\{- \frac{\theta(x,y,\tau)}{2 \ep}  \} } dy.
\label{e3.11}
\end{equation}
 Substituing the formulae \eqref{e3.9} and \eqref{e3.11} in \eqref{e3.10}, we get the expressions for $u^\ep$ and $V^\ep$ as given in \eqref{e3.4} and  also an expression for $v^\ep$. 
\end{proof}

\subsection{Large time asymptotic behavior of approximate solution}

In this section, we study the large time asymptotic behaviour of solutions  $(u^\epsilon,v^\epsilon)$
of  \eqref{e3.1} and \eqref{e3.2}, constructed in the previous section. First, we write the formula in the following way.
\begin{equation}
 \begin{aligned}
u^{\ep}= -2 \ep e^{ - \a t} \frac{  \frac{ \partial \tilde{S}^{\ep}}{\partial x}}{ \tilde{S}^{\ep}}, \,\,\,  
v^{\ep}= \frac{ \partial }{ \partial x} ( \frac{ \tilde{C}^{\ep} }{ \tilde{S}^{\ep}} )    
\end{aligned}
\label{e3.12}
\end{equation}
where
\begin{equation}
 \begin{aligned}
&\tilde{C}^{\ep}(x, \tau)= \frac{1}{ \sqrt{4 \pi \tau \ep }} \int_ \mathbb{R} { V_0(y)  \exp\{- \frac{\theta(x,y,\tau)}{2 \tau \ep}  \}  }dy\\
&\tilde{S}^{\ep}(x, \tau)= \frac{1}{ \sqrt{4 \pi \tau \ep }} \int_ \mathbb{R}{ \exp\{- \frac{\theta(x,y,\tau)}{2 \tau \ep}  \} } dy
\end{aligned}
\label{e3.13}
\end{equation}
and
\begin{equation}
 \begin{aligned}
&\theta(x, y, \tau) = \frac{{(x-y)}^2}{ 2\tau(t)} +U_0(y),\,\,\tau(t)=\frac{1-e^{- \a t}}{\a},\,\,U_0(y)= \int_0^y{ u_0(s) ds} ,\,\, V_0(y)=\int_0^y v_0(s) ds. 
\label{e3.14}
\end{aligned}
\end{equation}
The asymptotic form of the solution depends on the sign of $\a$ because
\begin{equation}
\lim_{t\rightarrow \infty}\tau(t)= \begin{cases} \displaystyle
      {1/\a,\,\,\, if \,\,\,\a>0,} 
\\\displaystyle
        { \infty 
                      ,\,\,\, if \a<0 }.
\end{cases}
\label{e3.15}
\end{equation}
We introduce the following functions to describe the asymptotic form.
\begin{equation}
\begin{aligned}
G(x)= \int_ \mathbb{R}{  \exp\{- \frac{1}{2 \ep }  \theta(x, y,  \frac{1}{\alpha})\} dy }, \,\,\, \,\,\,
H(x) = \int_ \mathbb{R}{  V_0(y)  \exp\{- \frac{1}{2 \ep }  \theta(x, y,  \frac{1}{\alpha})\} dy },
\end{aligned}
\label{e3.16}
\end{equation}
and
\begin{equation}
\begin{aligned}
G_\infty(\xi)&=
e^{-{U_0(+\infty)\over\epsilon}}\int_{-\infty}^{\xi }
e^{-y^2/2} dy+
e^{-{U_0(-\infty)\over\epsilon}}\int_{\xi}^{\infty}
e^{-y^2/2} dy, \\
H_\infty(\xi) &=  
e^{-{U_0(+\infty)\over\epsilon}} V_0(\infty)\int_{-\infty}^{\xi }
e^{-y^2/2} dy+
e^{-{U_0(-\infty)\over\epsilon}}V_0(-\infty)\int_{\xi}^{\infty}
e^{-y^2/2} dy,
\end{aligned}
\label{e3.17}
\end{equation}
where $\xi = \frac{x}{ \sqrt{2 \ep \tau(t)}}$.

\begin{theorem}
Let $u_0 \in L^\infty( \mathbb{R})\cap L^1( \mathbb{R})$ and $v_0$ is a function of bounded variation which is in $L^1( \mathbb{R})$, then we have the following asymptotic form for the solution $(u^\epsilon,v^\epsilon)$ of \eqref{e3.1} with initial condition \eqref{e3.2}.\\
 Case 1: When $\alpha>0$
\begin{equation}
\begin{aligned}
&\lim_{t \rightarrow \infty}\frac{-1}{ 2 \ep} e^{\alpha t} u^{\ep}(x, t) = \frac{ \frac{ d G }{ d x} }{G}, \,\,
&  \lim_{t \rightarrow \infty} v^{\ep}(x, t) = \frac{d }{ d x} ( \frac{ H(x) }{G(x)} ) 
\end{aligned}
\label{e3.18}
\end{equation}
and this limit is uniform with respect to $x$ on compact subsets of $\mathbb{R}$.\\
 Case 2: When $\alpha<0$
\begin{equation}
\begin{aligned}
\lim_{t \rightarrow \infty}\frac{-\sqrt{\tau(t)}}{ \sqrt{2 \ep} }e^{\alpha t} u^{\ep}(x, t) = \frac{ \frac{dG}{d\xi} }{G}, \,\,\,
 \lim_{t \rightarrow \infty} \sqrt{2 \epsilon \tau(t)}v^{\ep}(x, t) =  \frac{d}{d\xi}( \frac{ H(\xi)}{G(\xi)} )
\end{aligned}
\label{e3.19}
\end{equation}
and this limit is uniform with respect to $\xi$ on compact subsets of $\mathbb{R}$.
\end{theorem}
\begin{proof}
 If $\alpha >0$, from \eqref{e3.4}, we have $\tau  \rightarrow \frac{1}{\alpha}$ as $t\rightarrow \infty$. It follows  from the formula for $\tilde{S}^\ep$  and $\tilde{C}^\ep$ given by  \eqref{e3.12} and \eqref{e3.13} and from the formula \eqref{e3.16} for $G$ and $H(x)$, that
\begin{equation}
\begin{aligned}
\lim_{t\rightarrow \infty}\tilde{S}^{\ep}(x, \tau) &=  \tilde{S}^{\ep}(x,   \frac{1}{\alpha} ) = \frac{\alpha}{\sqrt{4 \pi {\ep}} } G(x) \\
\lim_{t\rightarrow \infty}\partial_x \tilde{S}^{\ep}(x, \tau) &=  \partial_x \tilde{S}^{\ep}(x,   \frac{1}{\alpha} ) = \frac{\alpha}{\sqrt{4 \pi {\ep}} } \frac{dG(x)}{dx}\\
\lim_{t \rightarrow \infty }\tilde{C}^{\ep}(x, \tau)& = \tilde{C}^{\ep}(x,   \frac{1}{\alpha} ) 
= \frac{\alpha}{\sqrt{4 \pi {\ep}} }H(x) \\
\lim_{t \rightarrow \infty }\partial_x \tilde{C}^{\ep}(x, \tau)& =\partial_x \tilde{C}^{\ep}(x,   \frac{1}{\alpha} ) 
= \frac{\alpha}{\sqrt{4 \pi {\ep}} }\frac{dH(x)}{dx} 
\end{aligned}
\label{e3.20}
\end{equation}
These limits are uniform with respect to $x$, where $x$ belongs to compact subsets of $\mathbb{R}$.
Substituing the asymptotic formula \eqref{e3.20} in the formula for $(u^\epsilon,v^\epsilon)$ given by \eqref{e3.12}-\eqref{e3.13}, we get \eqref{e3.18}.

Now, we consider the case when $\alpha<0$. We note that a typical term that we need to consider is of the form
\[
\Phi(x,t)=\frac{1}{ \sqrt{4 \pi \tau \ep }} \int_{\mathbb{R}} { \phi(y) \exp\{- \frac{\theta(x,y,\tau)}{2 \tau \ep}  \}  }dy
\]
with $\phi(y)=V_0(y)=\int_0^y v_0(s) ds$ or $\phi(y)=1$.  Setting $\xi = x/\sqrt{2 \epsilon  \tau(t)}$,  then making a change of
variable $z = \frac{\sqrt{2\epsilon \tau(t)} \xi - y}{\sqrt{2 \epsilon \tau(t)}}$ and
renaming $z$ as $y$, we get 
\begin{equation}
\Phi(x,t)=\frac{1}{ \sqrt{2 \pi  }} \int_{\mathbb{R}} { \phi(\sqrt{\epsilon \tau}(\xi-y) \exp\{- U_0(\sqrt{\epsilon \tau}(\xi-y)/\epsilon -y^2 \}  }dy.
\label{e3.21}
\end{equation}
To study the large time behavior of $(u^\epsilon,v^\epsilon)$, we follow Hopf \cite{Hopf}  by splitting the integral on the right hand side as the sum of integrals on three intervals , namely, $(-\infty,\xi-\delta), [\xi-\delta,\xi+\delta]$
and  $(\xi+\delta,\infty)$ where $\delta>0$ is small: 

\begin{equation*}
\begin{aligned}
\Phi(x,t)&=\frac{1}{ \sqrt{2 \pi  }} \int_{-\infty}^{\xi-\delta} { \phi(\sqrt{\epsilon \tau}(\xi-y) \exp\{- U_0(\sqrt{\epsilon \tau}(\xi-y)/\epsilon -y^2 \}  }dy\\
+&\frac{1}{ \sqrt{2 \pi  }} \int_{\xi-\delta}^{\xi+\delta}{ \phi(\sqrt{\epsilon \tau}(\xi-y) \exp\{- U_0(\sqrt{\epsilon \tau}(\xi-y)/\epsilon -y^2 \}  }dy\\
+&\frac{1}{ \sqrt{2 \pi  }} \int_{\xi+\delta}^\infty { \phi(\sqrt{\epsilon \tau}(\xi-y) \exp\{- U_0(\sqrt{\epsilon \tau}(\xi-y)/\epsilon -y^2 \}  }dy.
\end{aligned}
\end{equation*}

Using the assumptions of the theorem, the integral on $[\xi-\delta,\xi+\delta]=O(\delta)$. Now, passing  $t$ to  $\infty$ 
in the other two terms and then letting $\delta$ to zero, we arrive at
 
\begin{equation}
\begin{aligned} 
\lim_{t\rightarrow \infty}\Phi(x,t)
&= \frac{1}{ \sqrt{2 \pi  }} [
e^{-{U_0(+\infty)\over\epsilon}}\phi(\infty)\int_{-\infty}^{\xi }
e^{-y^2/2} dy+
e^{-{U_0(-\infty)\over\epsilon}}\phi(-\infty)\int_{\xi}^{\infty}
e^{-y^2/2} dy.]\\
\sqrt{2\epsilon \tau(t)}\Phi_x(x,t)&= \frac{1}{ \sqrt{2 \pi  }} \frac{d}{d\xi} [
e^{-{U_0(+\infty)\over\epsilon}}\phi(\infty)\int_{-\infty}^{\xi }
e^{-y^2/2} dy+
e^{-{U_0(-\infty)\over\epsilon}}\phi(-\infty)\int_{\xi}^{\infty}
e^{-y^2/2} dy.]
\end{aligned}
\label{3.22}
\end{equation}
We observe that, these limits are uniform for $\xi$, belonging to compact subsets of $\mathbb{R}$. Using \eqref{3.22} with  $\phi=1$ and $\phi(y)=\int_0^yv_0(z)dz$  in the formula \eqref{e3.12}, we arrive at
\begin{equation}
\begin{aligned}
\lim_{t\rightarrow \infty}\tilde{S}^{\ep}(x, \tau(t)) & = \frac{1}{\sqrt{2 \pi } } G(\xi) \\
\lim_{t\rightarrow \infty}\sqrt{2 \epsilon \tau(t)} . \partial_x \tilde{S}^{\ep}(x, \tau(t)) &= \frac{1}{\sqrt{2 \pi } } \frac{dG(\xi)}{d\xi}\\
\lim_{t \rightarrow \infty }\tilde{C}^{\ep}(x, \tau(t))&= \frac{1}{\sqrt{2 \pi} }H(x) \\
\lim_{t \rightarrow \infty }\sqrt{2 \epsilon \tau(t)}. \partial_x \tilde{C}^{\ep}(x, \tau(t))& = \frac{1}{\sqrt{2 \pi} }\frac{dH(x)}{dx}. 
\end{aligned}
\label{e3.23}
\end{equation}
Using \eqref{e3.12} in the formula for $(u^\epsilon, v^\epsilon)$ given by \eqref{e3.1}-\eqref{e3.3}, we get (3.19).
This completes the proof of the theorem.
\end{proof}

\section {\bf Weak  solution to  the Cauchy problem \eqref{e1.1}-\eqref{e1.2}.}

In this section, we work on the Cauchy problem 

\begin{equation}
 \begin{aligned}
&\frac{\partial u}{ \partial t} + \frac{\partial }{ \partial x} ( \frac{u^2}{2} ) + \a u = 0, \\
&\frac{\partial v}{ \partial t} + \frac{\partial }{ \partial x} (u v) =  0,
\end{aligned}
\label{e4.1}
\end{equation}
in the space - time domain $\{(x,t) : x \in \mathbb{R}, t>0\}$ with initial conditions
\begin{equation}
 \begin{aligned}
u(x, 0)= u_0(x), \,\,\, v(x, 0)= v_0(x), x \in \mathbb{R}.
\end{aligned}
\label{e4.2}
\end{equation}
With specific assumptions on the initial data, we construct the weak solution based on weak asymptotics, the Volpert product and Colombeau algebra. We use the approximate solution constructed in previous section in the construction.

\subsection{Weak asymptotic solution}
Before stating the results, we recall some notations introduced in previous section which appear in the analysis of solution.
Fix $(x, t)$ for  $x \in \mathbb{R}, t>0$ and given $u_0 \in L^{\infty}(\mathbb{R})$,  set 
\begin{equation}
 \begin{aligned}
\theta(x, y, \tau) = \frac{{(x-y)}^2}{ 2\tau(t)} + \int_0^y{ u_0(s) ds} , \,\,\,\tau(t)=\frac{1-e^{- \a t}}{\alpha}. 
\end{aligned}
\label{e4.3}
\end{equation}
With these notations, we state the main existence results of this paper.
 \begin{theorem}
 Let $u_0 \in L^\infty( \mathbb{R})$ and $v_0$  is in $L^1( \mathbb{R})$.  
\begin{itemize}
\item Let $u_0^\epsilon(x)=u_0*\eta^\epsilon(x), v_0^\epsilon(x)=v_0*\eta^\epsilon(x)$, where $\eta^\epsilon$ is the standard Friedrichs mollifier and  $(\tilde{u}^\ep, \tilde{V}^\ep)$ be defined by \eqref{e3.4}  with initial data $(u_0,v_0)$ replaced by $(u_0^\epsilon,v_0^\epsilon)$ . Then  with $\tilde{v}^\ep=\partial_x \tilde{V}^\ep$, $(\tilde{u}^\ep,\tilde{v}^\ep)$ is a weak asymptotic solution to \eqref{e1.1} with initial data \eqref{e1.2}.
\end{itemize}
\begin{itemize}
 \item For each fixed $t>0$, $\tau=\tau(t)>0$, except for a countable number of points $x \in  \mathbb{R}$, there exist a unique minimizer for the minimization problem
\begin{equation}
\min_{-\infty < y <\infty} \theta(x,y,\tau).
\label{e4.4}
\end{equation}
We call this minimizer as $y(x,\tau)=y(x,\tau(t))$, then at such points $(x, t),$ the  limits 
 $V^{\ep}(x, t) \rightarrow V(x, t)$ ,  $u^{\ep}(x, t) \rightarrow u(x, t)$  as $\ep \rightarrow 0$  exist and $(u, v)$ is given by 
 \begin{equation}
\begin{aligned}
 u(x, t) =  e^{- \a t} \frac{ (x - y(x, \tau(t)))}{ \tau(t)}  , \,\,\, V(x, t)= \int_0^{y(x,\tau(t))}  v_0(s) ds.
\end{aligned}
\label{e4.5}
\end{equation}
 $V(x,t)$  is a function of bounded variation, $v^{\ep} \rightarrow v(x,t)=\partial_x V(x,t)=v_0(y(x,t) \partial_x y(x,t)$
in distribution  and $(u,v)$ is a generalized solution for  \eqref{e1.1} with initial condition \eqref{e1.2}. 
\end{itemize}
\end{theorem}
\begin{proof}
 First, we show that $(\tilde{u}^\epsilon,\tilde{v}^\epsilon)$ is a weak asymptotic solution to \eqref{e1.1}-\eqref{e1.2}. The pair of functions 
$(\tilde{u}^\epsilon,\tilde{v}^\epsilon)$ belongs to $C^\infty ( \mathbb{R} \times (0,\infty)$, which follows easily by Leibniz rule. Since $u^\epsilon=e^{-\alpha t} \hat{U}_x$, to get  the estimate for $u^\epsilon$, it is enough to get  the estimate for $\hat{u}=\hat{U}_x$. From \eqref{e3.5}, $\hat{u}$ solves the following initial value problem
\[
\frac{\partial   \hat{u}}{ \partial t} + \frac{e^{-\a t}}{2}(\frac{\partial  \hat{u}^2}{ \partial x } )  =  {\ep} e^{- \a t}  \frac{{\partial}^2  \hat{u}}{ \partial x^2},
\]
with initial conditions 
\[
\hat{u}(x,0)=u_0*\eta_\epsilon(x).
\]
 Applying the maximum principle to this differential equation for $\hat{u}$,  we get
\[
||\hat{u}||_{L^\infty( \mathbb{R}\times[0, \infty))}\leq ||u_0^\epsilon||_{L^\infty( \mathbb{R})}\leq ||u_0||_{L^\infty( \mathbb{R})}.
\]
Since $\tilde{u}^\epsilon=e^{-\alpha t} \hat{u}$,  for any $T>0$
\[
||\tilde{u}^\epsilon||_{L^\infty( \mathbb{R} \times [0,T))} \leq e^{|\alpha | T} ||\hat{u}||_{L^\infty( \mathbb{R}\times[0,\infty))} \leq e^{|\alpha| T}
||u_0^\epsilon||_{L^\infty( \mathbb{R})} \leq e^{|\alpha|T}||u_0||_{L^\infty( \mathbb{R})}.
\]
Also, from the formula for $\tilde{V}^\epsilon$ 
\[
||\tilde{V}^\epsilon||_{L^\infty( \mathbb{R} \times [0,\infty)}\leq ||V_0^\epsilon||_{L^\infty( \mathbb{R})}\leq ||\int_0^x v_0^\epsilon dx||_{L^\infty( \mathbb{R})}\leq ||v_0||_{L^1( \mathbb{R})}.
\]
Now consider any test function $\psi \in C_c^\infty(\mathbb{R})$ and the following integrals,
\begin{equation}
 \begin{aligned}
&\int_ \mathbb{R} (\frac{\partial \tilde{u}^\epsilon}{ \partial t} + \frac{\partial }{ \partial x} \frac{( \tilde{u}^\epsilon)^2}{2}  + \a \tilde{u}^\epsilon)\psi dx = \int_{ \mathbb{R}} {\ep} e^{- \a t}  \frac{{\partial}^2 \tilde{u}^\epsilon}{ \partial x^2} \psi(x) dx=\int_{ \mathbb{R}} {\ep} e^{- \a t}   \frac{{\partial}^2 \psi}{ \partial x^2} \tilde{u}^\epsilon(x,t) dx, \\
&\int_{ \mathbb{R}}(\frac{\partial \bar{v}^\epsilon}{ \partial t} + \frac{\partial }{ \partial x} (\tilde{u}^\epsilon \tilde{v}^\epsilon))\psi(x) dx =  \int_{ \mathbb{R}}{\ep} e^{- \a t}  \frac{{\partial}^2 \tilde{v}^\epsilon}{ \partial x^2}\psi(x)dx =-\int_{ \mathbb{R}}{\ep} e^{- \a t}  \frac{{\partial}^3 \psi(x)}{ \partial x^3} \bar{V}^\epsilon(x,t)dx.
\end{aligned}
\label{e4.6}
\end{equation}
Clearly, the  terms on the  extreme right of \eqref{e4.6} goes to zero as $\epsilon \rightarrow 0$ unifomly in $t \in [0,T]$, for every $T>0$. Also by the property of the convolution
\[
u_0^\epsilon -u_0 \rightarrow 0,\,\,v_0^\epsilon -v_0 \rightarrow 0
\]
in distribution as $\epsilon \rightarrow 0.$ This proves that $(\bar{u}^\epsilon,\bar{v}^\epsilon)$ is a weak asymptotic solution of \eqref{e1.1} and \eqref{e1.2}.

Next, we proceed to prove that $u^{\ep} \rightarrow u$ and $V^{\ep} \rightarrow V$ pointwise a.e as $\ep \rightarrow 0$. For this, we follow Hopf \cite{Hopf}. The main step is the following observation of Hopf, which follows
basically from Jensen's inequality.\\
 
For every $(x,t)$ with $x\in \mathbb{R}, t>0$, the minimizer in \eqref{e4.4} exists. It is possible to have several such minimizers $y(x,\tau(t))$ and $y(x_1,\tau(t))\leq y(x_2,\tau(t))$ for $x_1< x_2$ . Further, for a.e $(x,t)$ these minimizers are unique. More precisely, fixing $\tau>0,$ Hopf \cite{Hopf} proved that, there exists  a unique minimum for $y,$  say $y(x, \tau)$ except for a countable $x,$ so that, 
\begin{equation*}
 \begin{aligned}
\theta(x, y_0(x, \tau), \tau)  = \min_{y \in \mathbb{R}}{\theta(x, y, \tau)} 
\end{aligned}
\end{equation*} 
where $\tau= \tau(t)$.
At the point $(x,t)$ where the unique minimum $y(x, \tau)$ in \eqref{e4.4} exists, we use the following theorem in Evans \cite{Evans1}: \\
{ \it
Suppose functions $k, l:  \mathbb{R} \rightarrow   \mathbb{R}$ are continuous functions, such that $l$ grows atmost linearly and $k$ grows at least quadratically. Suppose,  there exists  a unique point $y_0$ such that, 
$k(y_0)=\min_{y \in  \mathbb{R}}{k(y)}$, then we have that, \\
\begin{equation}
lim_{\ep \rightarrow 0}{     \frac{  \int_{-\i}^{\i}{ l(y) \exp{ \frac{- k(y)}{\ep} } dy}     }{      \int_{-\i}^{\i}{  \exp{ \frac{- k(y)}{\ep} } dy}  } }=l(y_0).
\label{e4.7}
\end{equation}
}
A direct application of this result with $l(y)=\int_0^y v_0(s)ds$ and $k(y)=\theta(x,y,\tau)$ with fixed $(x,t)$ where there is a unique minimizer for $\theta(x,.,\tau)$, gives 
\[
\lim_{\epsilon \rightarrow 0} V^\epsilon(x,t)=V(x,t)=\int_0^{y(x,\tau(t))}v_0(s)ds.
\]
In order to prove convergence of $v^\ep$ in distribution to $\partial_x V$,  first we show that $V^\ep$ is bounded, independent of $(x,t)$. Then use the fact that, a sequence of functions uniformly bounded and  pointwise convergent a.e., converges  in distribution and so does also all its derivatives. Now to prove uniform boundedness, we use the fact that $v_0$ is an integrable function as follows:

\begin{equation*}
 \begin{aligned}
&|V^\ep(x,\tau)| =  \frac{    |  \int_ \mathbb{R}{ \exp\{ \frac{-1}{2 \ep} \theta(x, y, \tau)  \}  (\int_0^y{ v_0(s) ds})   dy} | }{  | \int_ \mathbb{R}{ \exp\{ \frac{-1}{2 \ep} \theta(x, y, \tau) \}   dy} |     }\\
& \leq  (\int_{- \i}^{\i}{ | v_0(s) | ds})   \frac{    |  \int_ \mathbb{R}{ \exp\{ \frac{-1}{2 \ep} \theta(x, y, \tau)  \}     dy} | }{  | \int_ \mathbb{R}{ \exp\{ \frac{-1}{2 \ep} \theta(x, y, \tau) \}   dy} |     }\\
& \leq  (\int_{- \i}^{\i}{ | v_0(s) | ds})
\end{aligned}
\end{equation*} 
Thus, as $ \ep \rightarrow 0$ we obtain, 
\begin{equation*}
\begin{aligned}
v^{\ep}= \frac{ \partial  }{ \partial x} V(x,\tau) \rightarrow   \frac{ \partial  }{ \partial x} (\int_0^{y(x, \tau)}{v_0(s) ds}) 
\end{aligned}
\end{equation*} 
 in distribution. Now, since $v_0$ is a  function of bounded variation  in $x$ and $y_0(x, \tau)$ is a function of bounded variation in $x$ for each $t>0$, it follows that $\int_0^{y_0(x, \tau)}{v_0(s) ds}$ is a function of bounded variation  in $x$. Applying the chain rule for functions of bounded variation, we have 
\begin{equation*}
\begin{aligned}
 \frac{ \partial  }{ \partial x} (\int_0^{y(x, \tau)}{v_0(s) ds}) =v_0(y(x,  \tau))  \frac{ \partial y(x,\tau)}{ \partial x}.
\end{aligned}
\end{equation*} 
Now, using  $\tau(t) = \frac{1- e^{-\a t}}{ \a }$ and setting $y(x,t)=y_0(x,\tau(t))$, we get  as $\ep \rightarrow 0$,
\begin{equation*}
\begin{aligned}
v^{\ep} \rightarrow v_0(y(x, t))  \frac{ \partial y(x,t) }{ \partial x}
\end{aligned}
\end{equation*}  
 in distribution. The proof of the convergence of $v^{\epsilon}$ as $\ep \rightarrow 0$ is completed. 

Coming to  $u^{\epsilon},$ we see that 
\begin{equation*}
\begin{aligned}
u^\epsilon(x,t)= - 2 \ep e^{-\alpha t} \frac{  \frac{ \partial \tilde{S}^{\ep}}{\partial x} }{ \tilde{S}^{\ep} }(x, \tau) = 
 e^{-\alpha t} \frac{  \int_ \mathbb{R} \frac{(x-y)}{\tau}{ \exp\{- \frac{\theta(x,y,\tau)}{2 \ep}  \} } dy  }{   \int_ \mathbb{R} { \exp\{- \frac{\theta(x,y,\tau)}{2 \ep}  \} } dy  }.
\end{aligned}
\end{equation*}

By a direct application of \eqref{e4.7} to this formula for $u^\epsilon$, we have 
\begin{equation*}
\begin{aligned}
\lim_{\ep \rightarrow 0}u^\epsilon(x,t)=  e^{-\alpha t} \frac{(x-y(x, \tau))}{\tau}
\end{aligned}
\end{equation*}
where $y(x, \tau)$ is such that $\theta(x, y, \tau) = \min_{y \in R}{ \theta(x, y, \tau) }$. Since $(u,v)$ is the distributional limit of the weakly asymptotic solution, by definition it is a generalized solution of the inviscid problem \eqref{e1.1} and \eqref{e1.2}.
\end{proof}
\subsection{Volpert product and distributional solution}
Suppose the initial data is more regular, then the generalized solution constructed in previous section is a distributional solution. 
\begin{theorem}
Let $u_0 \in L^\infty(R)$ and $v_0$ is a function of bounded variation in $R$. 
\begin{itemize}
\item  Then the vanishing viscosity limit  $(u,v)$  given in the previous theorem is a distributional solution of \eqref{e1.1}.
\item Further the velocity $u$ satisfies the entropy condition $u(x+,t) \leq u(x-,t)$.
\end{itemize}
\end{theorem}
\begin{proof}
 We prove that the limiting distributions $(u,v)$ solve the equation \eqref{e1.1} with the initial condition \eqref{e1.2}. 
First, we show $u$ solves the first equation in \eqref{e1.1} with initial condition $u(x,0)=u_0(x)$. For this, we need to show
for all $\phi \in C_c^1(\mathbb{R} \times[0,\infty))$,
\[
\begin{aligned}
\int_0^\infty \int_{-\infty}^\infty u(x,t) \phi_t(x,t)+\frac{(u(x,t)^2)}{2} \phi_x( x,t) & - \alpha u(x,t) \phi(x,t)  
 +\int_{-\infty}^\infty u_0(x)\phi(x,0) dx=0
\end{aligned} 
\]
This is a standard as demonstrated, by passing to the limit as $\epsilon$ tends to zero in the equation
\[
\begin{aligned}
\int_0^\infty \int_{-\infty}^\infty (u^\epsilon(x,t) \phi_t(x,t)+\frac{(u^{\epsilon}(x,t)^2)}{2} \phi_x( x,t) & - \alpha u^\epsilon(x,t) \phi(x,t) + {\ep} e^{- \a t}  u^\epsilon(x,t) \phi_{xx})dx dt\\
& +\int_{-\infty}^\infty u_0(x)\phi(x,0) dx=0.
\end{aligned} 
\]
 and an application of dominated convergence theorem.

The proof that $v$ satisfies the second equation in \eqref{e1.1},  follows  by  an argument as in LeFloch \cite{LeFloch}. 
For completeness, we give the details here. As in \cite{LeFloch, Joseph2, Joseph3}, first we  show
\begin{equation} 
\mu = V_t + \overline{u} V_x=0
\label{e4.8}
\end{equation}
as a measure. First, we consider $(x,t)\in S_c$ that is a point of approximate continuity of $(u,V)$ and compute the first derivatives of $u=e^{-\alpha t} (\frac{x-y(x,\tau(t)}{\tau(t)})$.  We have
\[
\begin{aligned}
u_t&=e^{-\alpha t}[-\alpha \frac{x-y(x,\tau(t))}{\tau(t)}-\frac{x-y(x,\tau(t))}{\tau(t)^2} e^{-\alpha t}-e^{-\alpha t}\frac{y_t(x,t)}{\tau(t)}]\\
u_x&=\frac{e^{-\alpha t}}{\tau(t)}[1-y_x(x,\tau(t)].
\end{aligned}
\]
Using these expressions in the equation for  $u$, namely
\[
u_t + (\frac{u^2}{2})_x+\alpha u=0,
\]
and writing $\tau$ instead of $\tau(t)$, we get
\[
e^{-2 \alpha t}(-\frac{(x-y(x,\tau))}{\tau^2}-\frac{\partial_{\tau}y(x,\tau)}{\tau}+
\frac{(x-y(x,\tau))}{\tau}\frac{(1-\partial_{x}y(x,\tau))}{\tau})=0.
\]
Simplifying this equation we have,
\begin{equation}
\partial_{\tau}y(x,\tau) + \frac{(x-y(x,\tau))}{\tau} \partial_{x}y(x,\tau)=0.
\label{e4.9}
\end{equation}
Since $V(x,t)=\int_0^{y(x,\tau)} v_0(z)dz$ an easy calculation gives,
\[
\partial_{t}V(x,t) + u \partial_{x}V(x,t)=(\frac{dv_0}{dx})(y(x,\tau) e^{-\alpha t}
\{\partial_{\tau}y(x,\tau) +  \frac{(x-y(x,\tau))}{\tau}\partial_{x}y(x,\tau)\}
\]
and in view of \eqref{e4.9}, we arrive at
\begin{equation}
\partial_{t}V + u \partial_{x}V=0,
\label{e4.10}
\end{equation}
 for all points $(x,t)\in S_c$.

Next, consider $(s(t),t) \in S_j$,  a point on the  curve of discontinuity $x=s(t)$ of $(u, V)$.  Note that, the Rankine Hugoniot  condition 
for the equation for $u$ gives 
\begin{equation}
\frac{ds}{dt}= \frac{u(s(t)+,t)+u(s(t)-,t)}{2}.
\label{e4.11}
\end{equation}
We consider the measure $\mu=\partial_{t}V(x,t) +\bar{ u} \partial_{x}V(x,t)$ of $\{(s(t),t)\}$ for  $(s(t),t)\in S_j$:
\begin{equation}
\begin{aligned}
\mu(\{s(t),t)\})&=-\frac{ds}{dt}( V(s(t)+,t)-V(s(t)-,t) )\\
&+\int_0^1(u(s(t)-,t)+\alpha (u(s(t)+,t)-u(s(t)-,t))d\alpha
(V(s(t)+,t)-V(s(t)-,t))\\
&=[-\frac{ds}{dt}+\frac{u(s(t)+,t)+u(s(t)-,t)}{2}](V(s(t)+,t)-V(s(t)-,t)).
\end{aligned}
\label{e4.12}
\end{equation}
Using \eqref{e4.11} in \eqref{e4.12}, we get 

\begin{equation}
\mu(\{s(t),t)\}=[\partial_{t}V +  \overline{u} \partial_{x}V](s(t),t)=0.
\label{e4.13}
\end{equation}
In view of \eqref{e4.10}  and \eqref{e4.13} we get
\[
\partial_{t}V +  \overline{u} \partial_{x}V=0,
\]
in the sense of measures, which is \eqref{e4.8}. Now, taking the distributional derivative of this equation with respect to $x$ and using $v=V_x$, we arrive at

\[
\partial_{t}v + \partial_{x}(\overline{u} v)=0.
\]

The assertion that $u$ satisfies the entropy condition, follows easily from the fact that $y(x,\tau(t))$ is a non-decreasing function of $x$ and hence $y(x-,\tau(t)) \leq y(x+,\tau(t))$.
This completes the  proof of the theorem.
\end{proof}
\subsection{Weak solution in the sense of Colombeau}

Consider $(u,v)$ where 
$u=(u^\epsilon(x,t))_{0<\epsilon <1}$,
 and 
$v=(v^\epsilon(x,t))_{0<\epsilon <1}$,   
with $u^\epsilon$ and $v^\epsilon$  as solutions of  
equation \eqref{e2.3}
\begin{equation}
\begin{aligned}
u^\epsilon_t + (\frac{{u^\epsilon}^2}{2})_x +\alpha u&=\frac{\epsilon}{2} 
u^\epsilon_{xx},\\
v^\epsilon_t + (u^\epsilon v^\epsilon)_x &=\frac{\epsilon}{2} 
v^\epsilon_{xx},
\end{aligned}
\label{e4.14}
\end{equation}
in $\{(x,t) : x \in R^1, t>0 \}$, supplemented with an 
initial condition at $t=0$
\begin{equation}
u^\epsilon(x,0) = u^\epsilon_{0}(x) , v^\epsilon(x,0) = 
v^\epsilon_{0}(x), 
\label{e4.15}
\end{equation}
where
$u_0=(u_0^\epsilon(x))_{0<\epsilon <1}$, 
$v_0=(v_0^\epsilon(x))_{0<\epsilon <1}$, 
are  in ${\mathcal G}(R^1)$, the 
algebra  of generalized functions of Colombeau. We assume that $u_0^\epsilon$
and $v_0^\epsilon$ are bounded $C^\infty$ functions of $x$  with the following
estimates for $j=0, 1, 2, \dots$
\begin{equation}
\begin{aligned}
||{\partial_x}^j u_0^\epsilon||_{L^\infty([0,\infty))}&= O(\epsilon^{-j})
||{\partial_x}^j v_0^\epsilon||_{L^\infty([0, \infty))}&= 
O(\epsilon^{-j})
\end{aligned}
\label{e4.16}
\end{equation}
and $u_0^\epsilon(x) \rightarrow  u_0(x)$, 
$v_0^\epsilon(x) \rightarrow  v_0(x)$
pointwise a.e. 
These conditions are satisfied, for example, if we take bounded 
measurable functions or integrable functions on $R^1$ and then take its convolution with
the Friedrichs mollifiers with scale $\epsilon$.
 We shall prove the following result.
\begin{theorem}
 Assume that $u_0= (u_0^\epsilon(x)_{0<\epsilon <1}$, 
$v_0=(\rho_0^\epsilon(x)_{0<\epsilon <1}$ are in 
${\mathcal G}(R^1)$, 
with the estimates \eqref{e4.16} as described before. Let
$(u^\epsilon, v^\epsilon)$ be given by the formula \eqref{e3.4}
with $(u_0(x),v_0(x))$ replaced by 
$(u_0^\epsilon(x),v_0^\epsilon(x))$ for $\epsilon>0$,
then $ u= (u^\epsilon)_{0<\epsilon <1}$ and
$ v= (v^\epsilon)_{0<\epsilon <1}$ are in ${\mathcal G}(\Omega)$
and $(u,v)$ is a solution to \eqref{e2.20} with initial condition
 $(u_0, v_0)=(u_0^\epsilon(x),v_0^\epsilon(x))_{0<\epsilon <1}$.
\end{theorem}
\begin{proof}
 To show that $u$ and $v$ are in ${\mathcal G}(\Omega)$  
we need to prove that 
$u^\epsilon$ and $v^\epsilon$ 
satisfies the estimate \eqref{e2.3}. For this, we need to rewrite the formula for $u^\epsilon$ and $v^\epsilon$. First, we observe that 
\begin{equation}
 \begin{aligned}
u^{\ep}= -2 \ep e^{ - \a t} \frac{  \frac{ \partial \tilde{S}^{\ep}}{\partial x}}{ \tilde{S}^{\ep}}, \,\,\,  
v^{\ep}= \frac{ \partial }{ \partial x} ( \frac{ \tilde{C}^{\ep} }{ \tilde{S}^{\ep}} ).   
\end{aligned}
\label{e4.17}
\end{equation}
where
\begin{equation}
 \begin{aligned}
&\tilde{C}^{\ep}(x, \tau)= \frac{1}{ \sqrt{4 \pi \tau \ep }} \int_ \mathbb{R} { (\int_0^y v_0(s) ds)  \exp\{- \frac{\theta(x,y,\tau)}{2  \ep}  \}  }dy\\
&\tilde{S}^{\ep}(x, \tau)= \frac{1}{ \sqrt{4 \pi \tau \ep }} \int_ \mathbb{R}{ \exp\{- \frac{\theta(x,y,\tau)}{2  \ep}  \} } dy.
\end{aligned}
\label{e4.18}
\end{equation}

Now, using the expression for $\theta (x,y,\tau) = \frac{(x-y)^2}{2 \tau} +\int_0^y u_0^\epsilon(z)$ and using the change of variable $(x-y)/(4 \tau \epsilon)^{1/2} =z$, we can write
\[
\tilde{S}^\epsilon(x,t)=
\frac{1}{\pi} \int_{R} e^{-(z^2 +\frac{1}{2 \ep} \int_0^{x-(4 \tau \ep)^{1/2} z} u_0^\ep(s) ds)} dz.
\]
 Taking derivative of this expression with respect to $x$, we have
\[
\frac{ \partial }{ \partial x} \tilde{S^\epsilon}(x,t)=-\frac{1}{2 \pi \ep} \int_{R} u_0^\ep( x-(4 \tau \ep)^{1/2}z)e^{-(z^2 + \frac{1}{2 \epsilon} \int_0^{x-(4 \tau \ep)^{1/2} z}u_0^\ep(s) ds)} dz.
\]
Using this in \eqref{e4.17}, we get

\begin{equation}
u^\epsilon(x,t)=e^{-\alpha t} \frac{ \int_{R} u_0^\ep( x-(4 \tau \ep)^{1/2}z)e^{-(z^2 + \frac{1}{2 \epsilon} \int_0^{x-(4 \tau \ep)^{1/2} z} u_0^\ep(s) ds)} dz}{ \int_{R} e^{-(z^2 + \frac{1}{2 \epsilon} \int_0^{x-(4 \tau \ep)^{1/2} z}u_0^\ep(s) ds)}dz}.
\label{e4.19}
\end{equation}
Similarly,
\begin{equation}
v^\epsilon(x,t)=\partial_x V^\ep(x,t),\,\,V^\ep(x,t)=\frac{ \int_{R} \int_0^{x-(4 \tau \ep)^{1/2} z} v_0^\ep(s) ds\,\, e^{-(z^2 + \frac{1}{2 \epsilon} \int_0^{x-(4 \tau \ep)^{1/2} z} u_0^\ep(s) ds)} dz}{ \int_{R} e^{-(z^2 + \frac{1}{2 \epsilon} \int_0^{x-(4 \tau \ep)^{1/2} z}u_0^\ep(s) ds)}dz}.
\label{e4.20}
\end{equation}
Note that this can be written as
\begin{equation}
\begin{aligned}
u^\epsilon(x,t)&=e^{-\alpha t}  \int_{R} u_0^\ep( x-(4 \tau \ep)^{1/2}z) d\mu_{x,t}^\ep(z) \\
v^\epsilon(x,t)&=\partial_x V^\ep,\,\, V^\ep= \int_{R} (\int_0^{x-(4 \tau \ep)^{1/2} z} v_0^\ep(s) ds) \,\,d\mu_{x,t}^\ep(z)
\end{aligned}
\label{e4.21}
\end{equation}
where $d\mu_{x,t}(z)$ is a probability measure in $z$ variable parametrized by $\{(x,t), x \in R, t \geq 0\}$:
\begin{equation}
d\mu_{(x,t)}^\ep(z)=\frac{ e^{-(z^2 + \frac{1}{2 \epsilon} \int_0^{x-(4 \tau \ep)^{1/2} z} u_0^\ep(s) ds)} dz}{ \int_{R} e^{-(z^2 + \frac{1}{2 \epsilon} \int_0^{x-(4 \tau \ep)^{1/2} z}u_0^\ep(s) ds)}dz}.
\label{e4.22}
\end{equation}

From the formulas of $u^\ep,v^\ep$, it follows from  Leibinitz's rule that $\partial_x ^k u^\ep$  are integrals with integrands that are polynomials of degree at most $k$  which are derivatives of $u_0^\epsilon$ up to order $k$ and  $\partial_x ^k v^\ep$  are integrals of the form \eqref{e4.21} with respect to the measure  \eqref{e4.22} with integrands being polynomials of degree at most $k+1$  which are derivatives of $u_0^\epsilon$ and $\int_0^x v_0^\ep(y) dy$ up to order $(k+1)$.

From above observations  and estimate \eqref{e4.16}, it is easy to see the estimates,
\begin{equation*}
\parallel \partial^{k}_{x} u^{\epsilon} \parallel_
{L^{\infty}(\Omega_T)} = {\mathcal O}(\epsilon^{-2 k}) ,
 \parallel \partial^{\alpha}_{x} v^{\epsilon} \parallel_
{L^{\infty}(\Omega_T)} = {\mathcal O}(\epsilon^{-2 k}). 
\end{equation*}
where $\Omega_T =\{(x,t) :x \in R, 0<t<T \}$, for $T>0$.

Similarly, applying the differential operator ${\partial_t} ^{j} {\partial_x} ^
{k}$ on both sides of \eqref{e4.19} and \eqref{e4.20} $k=1,2,3, \cdots; \,\,\, j=0,1,2, 
\cdots $
we get,
\[
\parallel \partial^{j}_{t} \partial^{k}_{x} u^{\epsilon} 
\parallel_
{L^{\infty}(\Omega_T)} = {\mathcal O}(\epsilon^{-2(j+k)})
\]
\[ 
\parallel \partial^{j}_{t} \partial^{k}_{x} v^{\epsilon} \parallel_
{L^{\infty}(\Omega_T)} = {\mathcal O}(\epsilon^{-2(j+k)}).
\] 
These estimates show that $u$ and $v$ are in ${\mathcal G}(\Omega)$.

Now, we  show that $u$ and $v$  satisfy the equation \eqref{e2.20} in the 
sense of
association. We multiply \eqref{e3.1} by a test function $\phi \in 
C_0^\infty(\Omega)$ and integrate  to get
\[
\int_{0}^{\infty}\!\!\int_{-infty}^\infty(u^\epsilon_t+ 
(1/2)({u^\epsilon}^2)_x 
   \phi\,dx\,dt=
\frac{\epsilon}{2}\int_{0}^{\infty}\!\!\int_{-infty}^\infty \
u^\epsilon 
\phi_{xx} \,dx\,dt
\]
\[
\int_{0}^{\infty}\!\!\int_{-\infty}^\infty(v^\epsilon_t+ (u^\epsilon 
v^\epsilon)_{x}\phi \,dx\,dt
=-\frac{\epsilon}{2}
\int_{0}^{\infty}\!\!\int_{-\infty}^\infty {V^\epsilon} 
{\phi}_{xxx}\,dx\,dt.
\]
We have to show that the right hand side goes to zero as $\epsilon$
goes to zero. This easily follows  by an
application of dominated convergence theorem,
as $u^\epsilon(x,t)$ and $V^\epsilon$ are 
bounded and 
converge pointwise almost every where. This completes the proof of
the theorem.
\end{proof}

\section{Conclusion}

The system of equations  \eqref{e1.1} is of extreme importance and  has a variety of applications including the analysis of the Eulerian droplet model for air particle flow.  The solutions of this system do not belong to the classical $L^\infty$ or the  space of functions of bounded variation, but generally are in the space of measures, even if we start with  smooth initial data. This is due to concentration of mass which leads to formation of $\delta$ waves in the density component. The particular case when $\alpha\neq 0$  has important applications, like in the modeling of the large scale structure of the universe and sticky particle dynamics and is well understood. For general $\alpha>0$, the system \eqref{e1.1} is not very well studied due to the complexities in the analysis.  As far as we know, the only work done on this system is for Riemann type of initial data. The significance of our work is in the consideration of  general initial data and not just initial data of the Riemann type. In this work,  the exact solution to the Cauchy problem for a parabolic approximation of \eqref{e1.1} with general  initial condition \eqref{e1.2} is studied by adding a viscous term with a small coefficient.  Using this, we construct the explicit weak asymptotic solution  for \eqref{e1.1} having general initial data and the weak solution in the sense of Colombeau. Later, we  derived the vanishing viscosity limit of the solution and showed that the limit satisfies the inviscid system in the sense of distributions. This regularization is shown to be justified, as the characteristic speed $u$ satisfies the Lax shock inequality. In this study we also obtained the large time behaviour of solution to the viscous system. \\

{ \flushleft

\section*{Acknowledgements} 

The author would like to thank the referee for the constructive comments and suggestions that has improved this paper. \\
\vspace{0.25 cm}
Data availability statement: None \\

Funding statement: None \\

Conflict of interest disclosure: None \\

Ethics approval statement: None \\

Patient consent statement: None \\

Permission to reproduce material from other sources: Not Applicable\\

Clinical trial registration: Not Applicable\\
}
                
\section*{ORCID}
Kayyunnapara Divya Joseph https://orcid.org/0000-0002-4126-7882

\end{small}


\begin{thebibliography}{99}

\bibitem{Albeverio}
\newblock Albeverio S. and  Shelkovich V. M.
\newblock \emph { On the delta-shock front problem, Analytical approaches to multidimensional balance laws}.
\newblock Nova Sci. Publ.
New York. (2006) 45--87. 


\bibitem{a1}
\newblock Armbruster D.,  Marthaler D., Ringhofer C.
\newblock \emph {  Kinetic  and fluid model hierarchies for supply chains.}
\newblock Multiscale Model Simul. { \bf 2}(1)  (2004)  43--61


\bibitem{a2}
\newblock Armbruster D., Matthew Wienke.
\newblock \emph {  Kinetic models and intrinsic timescales: Simulaneous comparison for a 2nd order queueing model.}
\newblock AIMS' Journal. {\bf 12}(1) (2019)  177--193. 




\bibitem{Bateman}
\newblock Bateman H.
\newblock \emph {Some recent researches on the motion of fluids.}
\newblock Monthly Weather Review.  { \bf 43}(4) (1915) 163--170.

\bibitem{Burgers}
\newblock Burgers J. M.
\newblock \emph {A mathematical model illustrating the theory of turbulence.}
\newblock Adv. Appl. Mech. { \bf 1}  (1948) 171--179.



\bibitem{Colombeau1}
\newblock Colombeau, J.F. 
\newblock \emph {New Generalized functions and Multiplication of
distributions.} 
\newblock North Holland, Amsterdam, (1984).

\bibitem{Colombeau2}
\newblock Colombeau, J. F. and Heibig, A.
\newblock \emph {Nonconservative products in bounded variation functions.}
\newblock SIAM Jl. Math. Anal. {\bf 23} (1992) no. 4, 941-949.


\bibitem{Colombeau3}
\newblock Colombeau, J. F., Heibig, A. and M. Oberguggenberger, 
\newblock \emph {The Cauchy problem in a space of generalized functions 1, 2,} 
\newblock C. R. Acad. Sci. Paris Ser. 1 Math.,
{\bf 317} (1993) 851-855; {\bf 319} (1994) 1179-1183.

\bibitem{Joseph3}
\newblock Das, Abhishek; Joseph, K. T.and Sahoo, Manas R.
\newblock \emph{Initial value problem for the nonconservative zero-pressure gas dynamics system,}
\newblock Adv. Pure Appl. Math. {\bf 12} (2021) 16-49.

\bibitem{Dalmaso}
\newblock Dal Maso,G., LeFloch,P.G. and Murat, F.
\newblock \emph{Definition and weak stability of 
nonconservative products,} 
\newblock J. Math. Pures Appl. {\bf 74} (1995) 483-548.

 \bibitem{Danilov}
\newblock Danilov, V.G.,G.A.Omelyanov, G.A. and Shelkovich,V.M.
\newblock \emph {Weak asymptotics method and interaction of nonlinear waves (asymptoticmethods for wave and quantum problems)},33-163, 
\newblock  American Math.Soc. Transl. Ser.2,208, American Math. Soc., Provodence RI, (2003).


\bibitem{Egrov}
\newblock  Egorov,Y.U.,
\newblock \emph {  On the theory of generalized functions.}
\newblock Uspekhi Mat. Nauk { \bf 45} (1990), no.5 (275) 3-40, 222, Translations in Russian mathematical surveys { \bf 45} (1990) no.5, 149.

\bibitem {Evans1} 
\newblock  Evans L. C. 
\newblock \emph { Partial differential equations, Chapter 4, Section 4.5.2. }
\newblock AMS. (1998).

 
\bibitem{Forestier}
\newblock  Forestier-Coste L., Gottlich S. and Herty M.
\newblock \emph {  Data- Fitted Second - order Macroscopic Production Models.}
\newblock Siam J. Appl. Math. { \bf 75}(3) (2015) 999--1014.



\bibitem{Glimm}
\newblock Glimm J.
\newblock  \emph { Solutions in the large for nonlinear systems of equations.}
\newblock Comm. Pure Appl. Math. { \bf 18} (1965);  697-715.


\bibitem{Hopf}
\newblock Hopf E.
\newblock  \emph { The partial differential equation $u_t +u u_x
=\mu u_{xx}.$ }
\newblock Comm. Pure Appl. Math. { \bf 3 } (1950) 201--230.


\bibitem{Joseph1}
\newblock Joseph K. T.
\newblock  \emph {A Riemann problem whose viscosity solution contain $\delta-$ measures,}
\newblock  Asymptotic Anal. (Math, Sci.) { \bf 7} (1993) 105-120. 


\bibitem{Joseph2}
\newblock Joseph K. T. and Vasudeva Murthy A.S.
\newblock  \emph { Hopf -Cole transformation to some systems of partial differential equation,}
\newblock NoDEA Nonlinear Differential equations Appl.  { \bf 8} (2001) 173-193.


\bibitem{Keita}
\newblock Keita S., Bourgault Y.
\newblock  \emph { Eulerian droplet model: delta-shock waves and solution of the Riemann problem, }
\newblock J. Math. Anal. Appl. { \bf 472}(1) (2019) 1001--1027.


\bibitem{Kruzkov}
\newblock Kruzkov S.
\newblock  \emph { First order quasilinear equations with several space variables, }
\newblock  Math.USSR Sbornik  { \bf 10} (1970), 217-273. 


\bibitem{Lax}
\newblock Lax P.D.
\newblock  \emph { Hyperbolic systems of conservation laws II }
\newblock Comm. Pure Appl. Math. { \bf 10} (1957), 537-567.

\bibitem{LeFloch}
\newblock LeFloch P. G.
\newblock  \emph { An existence and uniqueness result for two nonstrictly hyperbolic systems, in Nonlinear Evolution Equations that change type, (eds) Barbara Lee Keyfitz and Michael Shearer.}
\newblock IMA. 1990; 27:126-138.



\bibitem{Richard}
\newblock Richard De la cruz.
\newblock \emph {  Riemann Problem for 2 X 2 Hyperbolic System with Linear Damping.}
\newblock  Acta Appl Math. { \bf 170} (2020) 631--647.


\bibitem{Shandarin}
\newblock  Shandarin S. F., Zeldovich Y. B.
\newblock \emph {Large-scale structure of the universe: turbulence, intermittency, structures
in a self-gravitating medium. }
\newblock Rev. Mod. Phys. { \bf 61} (1989) 185--220.



\bibitem{Saichev}
\newblock Gurvatov S. N. and Saichev A. I.
\newblock \emph {New approximation in the  adhesion model  in the description of the large scale structure of the universe. }
\newblock Cosmic Velocity fields : Proc. 9th IAP Astrophysics. (1993).



\bibitem{Maslov}
\newblock Maslov V.P. 
\newblock \emph {  Three algebras corresponding to unsmooth solutions of quasilinear hyperbolic equations}
\newblock  Uspekhi Mat.Nauk vol. 35 no.2 (1980) 252-253.


\bibitem{Shelkovich}
\newblock Shelkovich V. M.
\newblock \emph {$\delta$ and $\delta'$ shocks waves types of singular solutions of systems of conservation laws and transport and concentration process. }
\newblock Russian Math. Surveys. { \bf 63}(3) (2008) 73-146.

\bibitem{Tan}
\newblock Tan D., Zhang T. and Zheng Y. 
\newblock \emph {Delta-shock waves as limits of vanishing viscosity for hyperbolic systems of conservation laws.}
\newblock Jl. Diff. Equations. { \bf 112} (1994) 1--32.


\bibitem{Volpert}
\newblock Volpert A. I.
\newblock \emph {The space BV and quasi-linear equations.}
\newblock Math USSR Sb. { \bf 2} (1967) 225--267.


 \end{thebibliography}
\end{document}